\documentclass[11pt, a4paper, reqno]{amsart}

\usepackage{amsfonts,amsmath,amsthm}
\usepackage{amssymb}
\usepackage{latexsym}

\usepackage[english]{babel}

\numberwithin{equation}{section}

\newtheorem{theorem}{Theorem}[section]
\newtheorem{proposition}[theorem]{Proposition}
\newtheorem{lemma}[theorem]{Lemma}
\newtheorem{corollary}[theorem]{Corollary}
\newtheorem{cor}[theorem]{Corollary}

\theoremstyle{definition}

\newtheorem{definition}[theorem]{Definition}
\newtheorem{example}[theorem]{Example}

\theoremstyle{remark}

\newtheorem{remark}[theorem]{Remark}

\newcommand{\eps}{\varepsilon}

\newcommand{\Id}{\mathrm{Id}}

\newcommand{\dopu}{{:}\allowbreak\ }

\newcommand{\rest}[2]{#1\raisebox{-0.3ex}{\mbox{$\mid_{#2}$}}}

\newcommand{\loglike}[1]{\mathop{\rm #1}\nolimits}

\newcommand{\sign}{\loglike{sign}}

\newcounter{abc}   
\newcounter{iiiii} 

\newenvironment{aequivalenz}
{\setcounter{iiiii}{0}
\begin{list}%
{{\rm (\roman{iiiii})}}
{\usecounter{iiiii}
\parsep=0pt plus 1pt
\topsep=1pt plus 2pt minus 1pt
\itemsep=1pt plus 2pt minus 1pt
\leftmargin=3\baselineskip \labelsep=.6\baselineskip
\labelwidth=2.4\baselineskip
\rightmargin 0pt}%
}
{\end{list}}

\newenvironment{statements}%
{\setcounter{abc}{0}
\begin{list}%
{{\rm (\alph{abc})}}
{\usecounter{abc}
\parsep=0pt plus 1pt
\topsep=1pt plus 2pt minus 1pt
\itemsep=1pt plus 2pt minus 1pt
\leftmargin=3\baselineskip \labelsep=.6\baselineskip
\labelwidth=2.4\baselineskip
\rightmargin 0pt}%
}
{\end{list}}

{\begin{list}%
{$\bullet$}
{\leftmargin=3\baselineskip
\labelsep=\baselineskip \labelwidth=2.5\baselineskip
\parsep=0pt plus 1pt
\topsep=1pt plus 2pt minus 1pt
\itemsep=1pt plus 2pt minus 1pt
\rightmargin 0pt}%
}
{\end{list}}

\newcommand{\bea}{\begin{eqnarray*}}
\newcommand{\eea}{\end{eqnarray*}}
\newcommand{\beq}{\begin{equation}}
\newcommand{\eeq}{\end{equation}}
\newcommand{\begsta}{\begin{statements}}
\def\endsta{\end{statements}}
\newcommand{\begaeq}{\begin{aequivalenz}}
\def\endaeq{\end{aequivalenz}}

\def\DP{Daugavet property}

\begin{document}

\title{The $p$-Daugavet property for function spaces}

\author{Enrique A. S\'anchez P\'erez and Dirk Werner}




\thanks{The first author was partially
 supported by a grant from the Generalitat Valenciana (BEST/2009/108)
 and a grant from the Universidad Polit\'ecnica de Valencia
 (PAID-00-09/2291). Support of the Ministerio de Educaci\'on y 
Ciencia, under project \#MTM2009-14483-C02-02 (Spain)
and FEDER is also
gratefully acknowledged.}

\address{Instituto
Universitario de Matem\'atica Pura y Aplicada,  Universidad
\linebreak
Polit\'ecnica de Valencia, Camino de Vera s/n, 46071 Valencia, Spain.}
\email{easancpe@mat.upv.es}

\address{Department of Mathematics, Freie Universit\"at Berlin,
Arnimallee~6, \qquad {}\linebreak D-14\,195~Berlin, Germany}
\email{werner@math.fu-berlin.de}

\begin{abstract}
A natural extension of the Daugavet property for $p$-convex
Banach function spaces and related classes is analysed. As an
application, we extend the arguments given in the setting of the
Daugavet property to show that no reflexive space falls into this class.
\end{abstract}

\maketitle

\thispagestyle{empty}

\section{Introduction}

A Banach space $B$ is said to have the \textit{Daugavet property} if 
every rank one operator $T\dopu B \to B$ satisfies the
\textit{Daugavet equation}
$$
\|\Id+T\|=1+\|T\|, 
$$
where $\Id$ is the identity map on $B$. In the
setting of  Banach function spaces, several recent papers have
analysed for which ones the Daugavet property is satisfied. The
main examples are $L^1(\mu)$ and $L^\infty(\mu)$ whenever $\mu$
does not have any  atoms (see \cite[Section~2]{ams2000}). 
Also, if $K$ is a compact Hausdorff topological space without
isolated points, the space $C(K)$ has the Daugavet property; this
space is a function space, but not a Banach function space in the
sense that we use in this paper.

Some modifications of the Daugavet property have also been 
introduced in order to provide a weaker property for extending 
 the techniques and results that arises
for spaces with the Daugavet property
to a bigger class of spaces (see for example
\cite{boyK,martin,maroik,oik}).
In this paper we consider what we call generalized function spaces
(to be defined at the end of the introduction), a
new class that includes  
$p$-convex Banach function spaces and $C(K)$-spaces,  whose 
definition is partially motivated by the one of a
$K$-representable space given in \cite{gue}. Our generalization of the
Daugavet property is the following.

\begin{definition}
A constant~$1$ $p$-convex (generalized) function space $X$ is said
to have the \textit{$p$-Daugavet property} if and only if for every rank
one operator $T\dopu X \to X$, the equation
$$
\sup_{f \in B_{X}} \| ( f^p + T(f)^p)^{1/p} \| = (1+
\|T\|^p)^{1/p}
$$
is fulfilled. We call this equation for an operator $T$ the
\textit{$p$-Daugavet equation}.
\end{definition}

We remark that we deal with a different $p$-version of the \DP\ in our
paper~\cite{EnrDir2}. 

The aim of this paper is to characterize the $p$-Daugavet 
property and provide a description  
of the structure of the spaces that satisfy it, as well as to exhibit
a class of spaces  
having the property. As applications, and
 following the same steps as for spaces with the
Daugavet property, we prove that no reflexive constant~$1$ $p$-convex
generalized  function space  has the
$p$-Daugavet property. From
the methodological point of view, we adopt the Banach lattice
aproach to the Daugavet property (see \cite{abraali0,abraali}) but
use the geometric tools that were introduced later
(\cite{ams2000}).

We use standard notation. All the Banach spaces considered in the paper are supposed to be defined over the reals.
Let $Y$ be a Banach space.  $B_Y$ and $S_Y$ are the (closed) unit ball and the
unit sphere, respectively. $Y^*$ denotes the dual space of $Y$.
The slice $S(y^*, \varepsilon)$ defined by $y^* \in
B_{Y^*}$ and $\varepsilon >0$ is given by
$$
S(y^*, \varepsilon)= \{y \in B_Y\dopu  \langle y, y^* \rangle \ge 1-
\varepsilon \}.
$$

Let $(\Omega, \Sigma, \mu)$ be a measure
space. A \textit{Banach function space} $X(\mu)$  is an
ideal of the space $L^0(\mu)$ of classes of measurable functions
 (the usual $\mu$-a.e.\ order is considered)
that is a Banach space with a lattice norm ${\| \cdot
\|}$ such that for every $A \in  \Sigma$ of finite measure, $\chi_A
\in X(\mu)$ (see \cite[Def.~1.b.17]{lint}). If the measure $\mu$
is fixed in the context we simply write $X$ for $X(\mu)$.
$X(\mu)^+$ stands for the positive cone of $X(\mu)$ and in
general, if $S \subset X(\mu)$, we will write $S^+$ for its
positive part $S \cap X(\mu)^+$.

Let us now recall two basic geometric properties of  Banach
lattices. If $X$ is a Banach lattice, it is called
\textit{$p$-convex} if there is a constant $K$ such that for each finite
sequence $(x_i)_{i=1}^n$ in $X$,
$$
\Bigl\| \Bigl( \sum_{i=1}^n |x_i|^p \Bigr)^{1/p} \Bigr\|_X
 \le K \Bigl(\sum_{i=1}^n \|x_i\|_X^p \Bigr)^{1/p}.
$$
An operator $T\dopu X \to F$ on a  Banach lattice  $X$ is 
\textit{$p$-concave}
if there is a constant $k$ such that for every sequence
$(x_i)_{i=1}^n$ in $X$,
$$
\Bigl(\sum_{i=1}^n \|T(x_i)\|_F^p \Bigr)^{1/p} \le  k \Bigl\| \Bigl(
\sum_{i=1}^n |x_i|^p \Bigr)^{1/p} \Bigr\|_X.
$$
The quantities $M^{(p)}(X)$ and $M_{(p)}(T)$ are respectively  the
best constants  in the above inequalities. If $T$ is the identity map,
we also say that $X$ is $p$-concave and denote the corresponding
constant by $M_{(p)}(X)$. Throughout the paper we assume that the
$p$-convexity constants of the spaces are equal to~$1$; we will
write that the space $X$ is constant~$1$ $p$-convex for short.

In order to extend these lattice notions to a bigger family of spaces, we introduce
the class of \textit{generalized function spaces} (g.f.s.\ for
short). Let $1 \le p<\infty$. Let $\Delta$ be a set and consider a
family of measure spaces $\{(A_\delta, \Sigma_\delta, \mu_\delta)\dopu 
\delta \in \Delta\}$. Let $(\mathbb{R}^{A_\delta})_{\mu_\delta}$
be the space of classes of $\mu_\delta$-a.e.\ equal measurable
functions, usually denoted by $L^0(A_\delta,\mu_\delta)$. 
For each element $f \in \prod_{\delta \in \Delta} (
\mathbb{R}^{A_\delta})_{\mu_\delta}$, the modulus $|f|$ 
is  defined pointwise by $|f|(\delta)=|f(\delta)|$. 
 Let $X$ be a Banach space of (classes of) functions $f \in \prod_{\delta \in \Delta} (
\mathbb{R}^{A_\delta})_{\mu_\delta}$ such that $|f| \in X$ and with a norm with lattice properties when the natural order inherited in $X$ from the product is considered. Thus, $X$ is a Banach space
of functions whose values at each point $\delta$ are classes of
$\mu_\delta$-a.e.\ equal real functions  
such that the modulus of each element belongs to the space, too, and with a lattice norm; notice that the counting
measure on $A_\delta$ is also admissible, which is important for
including $C(K)$ spaces. 
The measure $\mu_\delta$ will not be written explicitly if it is not
relevant in the context. 

Note that, if $f \in X$, we can write it at each point $\delta
\in \Delta$ as $f(\delta) = \sign \{f(\delta)\} |f(\delta)|$ and
define the map $i_p$ from $X$ taking values in $\prod_{\delta \in
\Delta} (\mathbb{R}^{A_\delta})_{\mu_\delta}$ by means of its pointwise evaluation 
$$
i_p(f)(\delta):= \sign \{f(\delta)\} |f(\delta)|^p,\quad f \in X.
$$
Here, $\sign \{a\}$ denotes the sign   of the real number~$a$.
We use the notation 
$$
f^p:=i_p(f);
$$
we caution the reader that for even integers $f^p$ need not be the
same as $|f|^p$. 
The map $i_p$ is
clearly an injection; we denote by $i_{1/p}$ (i.e.,
$i_{1/p}(g)=g^{1/p}$ for every $g \in i_p(X)$), the inverse
map.

We say that a Banach space of functions as above is a
\textit{constant~$1$ $p$-convex generalized function space} if for
every finite family of elements $x_1,\dots ,x_n \in X$ the function $
( \sum_{i=1}^n x_i^p )^{1/p} $ belongs to $X$ and
$$
\Bigl\| \Bigl( \sum_{i=1}^n |x_i|^p \Bigr)^{1/p} \Bigr\| 
\le \Bigl( \sum_{i=1}^n \|x_i\|^p \Bigr)^{1/p}.
$$

It can easily be seen that this class includes for instance
constant~$1$ $p$-convex Banach function spaces (take $\Delta$ a
singleton), $C(K)$-spaces and 
$C(K,Y)$-spaces, where $Y$ is a constant~$1$ $p$-convex Banach
function space.

\section{The $p$-convexification of the Daugavet property}

Let $X$ be a constant~$1$ $p$-convex g.f.s. Consider the linear
space $F_{X,p}$ of functions from $X$ to the corresponding product
$\prod_{\delta \in \Delta} (\mathbb{R}^{A_\delta})_{\mu_\delta}$
that are finite sums of the elements of the set
$$ S_{X,p}= \Bigl\{
\phi\dopu X \to \prod_{\delta \in \Delta}
(\mathbb{R}^{A_\delta})_{\mu_\delta}\dopu  \phi=i_p \circ T, \ T\dopu X
\to X \ \textrm{linear and continuous} \Bigr\}.
$$
Define the function norm on the space $F_{X,p}$  by
$$
\|\psi\|_{F_{X,p}}:=\sup_{f \in B_X} \Bigl\| \Bigl(\sum_{i=1}^n
\phi_i(f) \Bigr)^{1/p} \Bigr\|^p_{X}, \quad \psi=\sum_{i=1}^n \phi_i(f)  \in
F_{X,p}.
$$
Clearly, this formula is independent of the representation that is
used for the element $\phi$ as a sum of  elements of 
$S_{X,p}$. It can easily be  checked that it is a norm on the
 space $F_{X,p}$, just taking into account that
 ${\||\cdot|^{1/p}\|^p_X}$ is a norm in $i_p(X)$ since 
  $X$ is a constant~$1$ $p$-convex g.f.s.\ and that the map $i_p$ is 
one-to-one.

\begin{lemma}  \label{homoge}
Let $X$ be a constant~$1$ $p$-convex g.f.s. Suppose that $T$
and $S$ are operators from $X$ into $X$ such that
\begin{equation}\label{eq1}
\sup_{f \in B_X}\||T(f)^p+ S(f)^p|^{1/p}\|= (\|T\|^p+\|S\|^p)^{1/p}.
\end{equation}
Then for every $\alpha, \beta \ge 0$,
$$
\sup_{f \in B_X} \||(\alpha T(f))^p+ (\beta S(f))^p |^{1/p} \|= \big(
\|\alpha T\|^p+ \|\beta S\|^p \big)^{1/p}. 
$$
\end{lemma}

\begin{proof}
Consider the functions $\phi$ and $\varphi$ in $S_{X,p}$ defined
by $\phi= i_p \circ T$ and $\varphi=i_p \circ S$ and note that
$\|\phi\|_{F_{X,p}}=\|T\|^p$ and $\|\varphi\|_{F_{X,p}}=\|S\|^p$.
Therefore, by (\ref{eq1}) we have that
$$
\|\phi+\varphi\|_{F_{X,p}}= \|\phi\|_{F_{X,p}}+\|\varphi\|_{F_{X,p}}.
$$
This implies, by Lemma~11.4 in \cite{abraali} (or
\cite[p.~78]{Dirk-IrBull}), 
that for every couple of non-negative real numbers $a$ and $b$,
$$
\|a \phi+b \varphi\|_{F_{X,p}} = a \|\phi\|_{F_{X,p}}+b\|\varphi\|_{F_{X,p}}.
$$
But this can be rewritten as
$$
\sup_{f \in B_X} \||a (T(f))^p+ b(S(f))^p |^{1/p} \|^p_X= a \|T\|^p+ b\|S\|^p.
$$
Thus, the result holds just by considering $\alpha=a^{1/p}$ and $\beta=b^{1/p}$.
\end{proof}

\begin{proposition} \label{forone}
 Let $X$ be a constant~$1$ $p$-convex g.f.s.  Suppose that $T\dopu
 X\to X$ is the rank one operator 
given by $T(f):=\langle f, g^* \rangle g$, where $g^* \in X^*$ and
$g \in X$. The following statements are equivalent:
\begin{itemize}
 \item[(1)]  $\sup_{f \in B_X} \||f^p+ T(f)^p |^{1/p} \|_X= (1 + \|T\|^p)^{1/p}.$
\item[(2)] For every $\varepsilon >0$ there is an element $h \in S(\frac{g^*}{\|g^*\|},\varepsilon)$ such that
$$
\Bigl\| \Bigl| \Bigl(\frac{g}{\|g\|} \Bigr)^p + h^p \Bigr|^{1/p}
\Bigr\|^p \ge 2- 2\varepsilon. 
$$
\end{itemize}
\end{proposition}

\begin{proof}
Let us prove first that (1) implies (2). By  Lemma~\ref{homoge},
we can assume that the norm of $T$ is one, just by replacing $T$
by $T/\|T\|= T/\|g\|\|g^*\|$, and representing the resulting
operator with two norm one  elements that we still denote by $g$
and $g^*$. Let $\varepsilon >0$. Take $h \in S_X$ such that
$$
\||h^p+ T(h)^p |^{1/p} \|^p_X\ge 2-\varepsilon.
$$
We can assume that $\langle h,g^* \rangle \ge 0$ (otherwise,
replace $h$ by $-h$). Notice first that since $X$ is constant~$1$
$p$-convex, 
$$
1+ \langle h, g^* \rangle^p\|g\|_X^p = \| |h^p|^{1/p} \|^p_X + \|
|T(h)^p|^{1/p}\|^p_X \ge 2-\varepsilon, 
$$
which implies that $\langle h, g^* \rangle \ge (1-
\varepsilon)^{1/p} \ge 1-\varepsilon$. Consequently, $g \in
S(g^*,\varepsilon)$. On the other hand, also by the constant~$1$
$p$-convexity of $X$,
\bea
2-\varepsilon &\le& 
\| |h^p+T(h)^p|^{1/p}\|^p_X \\
&\le& 
\| |h^p+g^p|^{1/p} \|^p_X + \| |T(h)^p-g^p|^{1/p} \|^p_X \\
&=& 
\| | h^p+g^p|^{1/p} \|^p_X + \| |(1-\langle h,g^* \rangle^p)g^p|^{1/p} \|^p
\\
&\le& 
\| | h^p+g^p|^{1/p} \|^p_X + (1-\langle h,g^* \rangle^p) \|
|g^p|^{1/p} \|^p \\
&\le& 
\| | h^p+g^p|^{1/p} \|^p_X +  \varepsilon.
\eea
 This gives the result.

For the converse, first notice that
 the inequality 
$$
\sup_{f \in B_X} \||f^p+ T(f)^p |^{1/p} \|^p_X \le 1 + \|T\|^p
$$
always holds, by the constant~$1$ $p$-convexity of $X$. By
Lemma~\ref{homoge}, we can assume that $\|g\|=1$ and $\|g^*\|=1$,
and then $\|T\|=1$. Let $\varepsilon >0$ and $h \in
S({g^*},\varepsilon)$ such that
$$
\| |g^p + h^p |^{1/p}\|^p \ge 2- 2\varepsilon.
$$
Then, again by the constant~$1$ $p$-convexity of $X$,
\bea
2- 2\varepsilon \le \| |g^p + h^p |^{1/p}\|^p &=&
\| |g^p-T(h)^p + T(h)^p + h^p |^{1/p}\|^p\\
&\le&
\| |g^p-T(h)^p|^{1/p}\|^p + \||T(h)^p + h^p |^{1/p}\|^p \\
&\le& 
(1- \langle h, g^* \rangle^p) + \||T(h)^p + h^p |^{1/p}\|^p\\
&\le&
(1-(1-\varepsilon)^p)  + \||T(h)^p + h^p |^{1/p}\|^p.
\eea
Since this holds for every $\varepsilon >0$, we obtain the result.
\end{proof}

\begin{example}
Let $X(\mu)$ be a constant~$1$ $p$-convex Banach function space. 
Consider the set $P$ of positive rank one operators
from $X(\mu) \to X(\mu)$, i.e.,
$$
P=\{T\dopu X \to X \dopu 
T=g^* \otimes g, \ g^* \in (X(\mu)^*)^+, \ g \in (X(\mu))^+ \}.
$$
Proposition~\ref{forone} gives directly the
following result by taking into account that the supremum in (1) of
Proposition~\ref{forone} can be 
computed using just positive elements.

 The following assertions are equivalent:
\begin{itemize}
 \item[(1)] For every positive rank-one operator $T \in P$,
$$
\sup_{f \in B_X} \||f^p+ T(f)^p |^{1/p} \|_X= (1 +
\|T\|^p)^{1/p}.
$$
\item[(2)] For every $\varepsilon >0$, every $g \in S_{X}^+$ and every $g^* \in S_{X^*}^+$ there is an element
$h \in (S(g^*,\varepsilon))^+$ such that
$$
\| |g^p + h^p |^{1/p}\|^p \ge 2- 2\varepsilon.
$$
\end{itemize}

For instance, $L^p$-spaces satisfy
 the statements above, since they are $p$-concave and $M_{(p)}(L^p)=1$. This includes the case of $\ell^p$; recall
that the Daugavet property is not satisfied for $\ell^1$, so the
property given by the equivalent assertions above, at least for
the case $p=1$, is strictly weaker than the Daugavet property. We
will show that this  is also the case for $p>1$ but in a more
dramatic sense, since $L^p(\mu)$ over an atomless measure $\mu$ does
not satisfy the $p$-Daugavet property. The reader can find more
information about what is called the positive Daugavet property in
\cite[Section~5]{bilik}.

\end{example}

Besides this example and taking into account the purpose of this
paper, the main application of
Proposition~\ref{forone} is the geometric
characterization of the $p$-Daugavet property that is given in
the following result.

\begin{corollary}
Let $X$ be a constant~$1$ $p$-convex g.f.s.  The following
are equivalent:
\begin{itemize}
 \item[(1)] $X$ has the $p$-Daugavet property.
\item[(2)] For every $\varepsilon >0$, every $g \in S_X$ and every
  $g^* \in S_{X^*}$ there is an element $h \in S(g^*,\varepsilon)$
  such that 
$$
\| |g^p + h^p |^{1/p}\|^p \ge 2- 2\varepsilon.
$$
\end{itemize}
\end{corollary}

We now define a new class of spaces that we call
$(p,\mathcal{K})$-representable spaces. This definition
generalizes in a sense the one given in \cite[Definition~2.3]{gue}.

\begin{definition} \label{prep}

Let $1 \le p < \infty$ and let $I$ be an index set. Consider a family
$\mathcal{K}=\{K_i\dopu  i \in I\}$ of (disjoint) compact Hausdorff
spaces. Let $X$ be a Banach space. We say that
$X$ is \textit{$(p,\mathcal{K})$-representable} if there exists a  family
$(X_k)_{k \in \bigcup_{i \in I} K_i}$ of constant~$1$ $p$-convex
Banach function spaces or $C(K)$-spaces such that:

\begin{itemize}
\item[(i)]
Each $x \in X$ can be identified linearly with its coordinates in
the product
$\prod_{k \in \bigcup_{i \in I} K_i} X_k$, and if $i \in I$, 
the restriction of $x$ to the
product $\prod_{k \in K_i} X_k$ belongs to $X$. 
Also, for every finite family $x_1,\dots ,x_n \in X$, the element 
$(\sum_{\nu=1}^n x_\nu^p)^{1/p}$ that is defined pointwise
by means of its representation belongs to $X$. 

\item[(ii)]
Consider the space $\bigoplus_{i \in I}^{ \infty} C(K_i)$, where the
sup norm for the sum is considered. If $x \in X$,
 for every $(\varphi_i)_{i \in I} \in \bigoplus_{i
\in I}^{ \infty} C(K_i)$ the product $(\varphi_i)x= (\varphi_i
\rest xi)$ belongs to $X$.

\item[(iii)] For every $x \in X$,
$$
\|x\| = \sup \Bigl\{ \Bigl(\sum_{j\in F} \|x(k_{j})\|^p \Bigr)^{1/p}\dopu  
F\subset I \ \textrm{finite} , \ k_{j} \in K_{j}
 \Bigr\} <\infty .
$$

\item[(iv)]
 For every $x \in X$, $i \in I$ and $\varepsilon >0$, the set 
$$
\{k \in K_i\dopu  \|x(k)\| \ge (1- \varepsilon) \|\rest xi\| \}
$$ 
is infinite.

\end{itemize}

\end{definition}

\begin{remark} \mbox{}
\begin{itemize}
\item[(1)]
 If $F$ is a finite subset of $I$, we write $\rest xF$ for the
element $\rest xF= \sum_{i \in F} \rest xi$ that coincides with the
projection of $x$ in the coordinates belonging to $\prod_{k \in
\bigcup_{i \in F} K_i} X_k$.

\item[(2)] Notice that since each space $X_k$ has lattice
properties, by the description of the norm required in (iii) 
for every $(\varphi_i) \in B_{\bigoplus_{i \in I}^{\infty} C(K_i)}$ and $x \in X$,
$\|\varphi x\| \le \|x\|$.

\item[(3)] Straightforward calculations show that if the space $X$ is
  $(p, \mathcal{K})$-representable, then 
it is a constant~$1$ $p$-convex g.f.s.

\item[(4)]
The chief example of a $(p, \mathcal{K})$-representable space is the
$\ell^p$-sum \linebreak 
$\bigoplus_{i\in I}^p L^\infty (\mu_i)$; here $\mu_i$ is a positive
nonatomic measure. Recall that every space $L^\infty
(\Omega,\Sigma,\mu)$ is lattice isometrically isomorphic to some
$C(K)$ with $K$ a compact Hausdorff space \cite[p.~104]{Sch}, and if
$\mu$ is nonatomic, then $K$ is perfect so that (iv) of
Definition~\ref{prep} holds.   
Note that, by construction,  $\bigoplus_{i\in I}^p L^\infty(\mu_i)$ is
a constant-$1$ $p$-convex g.f.s.
\end{itemize}
\end{remark}

\begin{proposition}\label{prop2.8}
If $X$ is $(p,\mathcal{K})$-representable, then it has the
$p$-Daugavet property.
\end{proposition}

\begin{proof}
Take two elements $x \in S_X$ and $x^* \in S_{X^*}$ and
$\varepsilon >0$. There is a finite set $F_1 \subset I$ such that
$$
(1 -\varepsilon)^{1/2p} \|x\| \le    \Bigl(\sum_{i \in F_1} \|x|_{i}\|^p
\Bigr)^{1/p}.
$$

On the other hand, 
the set $\{\rest xF\dopu  x \in B_X, \ F \subset I$ finite$\}$
is clearly dense in $B_X$ and hence norming for~$X^*$.
Consequently, 
 there is a finite set $F_2 \subset I$ and an
element $z \in S_X$ such that
$$
\Bigl(1 -\frac{\varepsilon}{2}\Bigr)\|x^*\| \le \langle \rest z{{F_2}} , x^*
\rangle.
$$
Take $F= F_1 \cup F_2$ and $N=|F|$. As in the proof of Lemma~2.4
in \cite{gue}, for each $i \in F$, (iv) in Definition~\ref{prep}
provides a sequence of different points $(w_{i,n})_{n=1}^\infty$
in $K_i$ such that for each $n$,
$$
\|\rest xi(w_{i,n})\| > (1- \varepsilon)^{1/2p} \|x|_i\|.
$$
An application of Uryson's lemma in each $K_i$ provides a sequence
of normalised
disjointly supported functions $(f_{i,n})_{n=1}^\infty$ such
that $f_{i,n}(w_{i,n})=1$ for every $n$. For every $i \in F$, the
sequence $(f_{i,n})_n$ converges pointwise to $0$, and consequently it
converges weakly to $0$ in $C(K_i)$. Therefore, the function
$(f_{i,n})\dopu  \prod_{i \in F} K_i \to \mathbb{R}^N$ defined in each
coordinate as $f_{i,n}(w)$ converges weakly to $0$ in $\bigoplus_{i
\in F}^\infty C(K_i)$.

As a consequence of the requirements in Definition~\ref{prep}, for
each $v \in X$, the linear map $J\dopu \bigoplus_{i \in F}^\infty C(K_i) \to X$
given by $J((\varphi_i))= \sum_{i \in F} \varphi_i \rest vi$ is
well-defined and continuous, and so weak-to-weak continuous. This
implies that the sequence $(y_n)_{n=1}^\infty$, where
$$
y_n:=\sum_{i \in F} f_{i,n} \rest xi + \sum_{i \in F_2} (1- f_{i,n})
\rest zi
$$
(note that only $F_2$ appears in the second sum) converges weakly
to $\rest z{F_2}$ in~$X$. Notice that all these elements have norm
 ${\le1}$. Therefore there is an index $m$ such that
$\langle y_m, x^* \rangle > 1- \varepsilon$. Take $y:=y_m$.
Finally, note that
\bea
\||x^p+ y^p|^{1/p} \| 
&\ge&
\Bigl( \sum_{i \in F} \||x|_i^p+ y|^p_i|^{1/p} \|^p \Bigr)^{1/p} 
\allowdisplaybreaks \\
&\ge& 
\Bigl( \sum_{i \in F} \||x(w_{i,m})^p+ y(w_{i,m})^p|^{1/p} \|^p
\Bigr)^{1/p} 
\allowdisplaybreaks \\
&=& 
2^{1/p} \Bigl( \sum_{i \in F} \|x(w_{i,m})\|^p \Bigr)^{1/p} 
\allowdisplaybreaks \\
&\ge& 
2^{1/p}(1- {\varepsilon})^{1/2p} \Bigl( \sum_{i \in F} \|x|_i\|^p
\Bigr)^{1/p} 
\allowdisplaybreaks \\
&\ge& 
2^{1/p} (1- \varepsilon)^{1/p}.
\eea
This completes the proof.
\end{proof}

\begin{cor}
The space $\bigoplus_{i\in I}^p L^\infty(\mu_i)$ has the $ p$-\DP.
\end{cor}

\begin{remark}
A direct consequence of the definition of
$(p,\mathcal{K})$-representable spaces is that every representable
space in the sense of 
\cite[Definition~2.3]{gue}, 
which is represented on a compact set
$K$ and over a product 
$\prod_{k \in K} X_k$ of constant~$1$ $p$-convex Banach function
spaces or $C(K)$-spaces and for which the requirements in (i) of
Definition~\ref{prep} are fulfilled, is in fact
$(p,\mathcal{K})$-representable; it is enough to consider a one point
set $I$.
\end{remark}

\section{The $p$-Daugavet equation for weakly compact operators}

The proofs of the following results hold by adapting the
techniques used in the ones for the case of spaces with the
Daugavet property, so we only sketch the parts that are different.
For the \DP, Theorem~\ref{reflexive} was first proved in
\cite{ams2000} and Proposition~\ref{uncbasis} in \cite{kadets}. 

\begin{theorem} \label{reflexive}
Let $X$ be a constant~$1$ $p$-convex g.f.s.\  and with the
$p$-Daugavet property. Then every weakly compact operator
satisfies the $p$-Daugavet equation. Consequently, $X$ cannot be
reflexive.
\end{theorem}

\begin{proof}
Let us show first the following claim: 
\begin{itemize}
\item[$\bullet$]
\textit{Let $X(\mu)$ be a
constant~$1$ g.f.s. If for every $\varepsilon >0$ there are a
slice $S(g^*, \delta)$ and an element $g \in S_X$ such that
$T(S(g^*,\delta))$ is included in the ball $B_\varepsilon(g)$, then
$T$ satisfies the $p$-Daugavet equation.}
\end{itemize}

In order to see this, note that we can assume that $\|T\|=1$. Take
$\varepsilon>0$, and note that we can also assume that $0<\delta
\le \varepsilon$. By Proposition~\ref{forone} there is an element
$h \in S(g^*, \delta)$ such that 
$\| (h^p+g^p)^{1/p} \|_X \ge 2-2 \varepsilon$; it follows that 
$\|T(h)-g\|_X \le \varepsilon$.
Thus
\bea
\sup_{f \in B_X} \| (f^p+ T(f)^p)^{1/p} \|^p_X 
&\ge& 
\| (h^p + T(h)^p)^{1/p}\|^p_X 
\allowdisplaybreaks\\
&\ge& 
\| (h^p + g^p - (g^p- T(h)^p))^{1/p} \|_X^p 
\allowdisplaybreaks\\
&\ge&
 \| (h^p + g^p)^{1/p}\|^p_X - \|(g^p- T(h)^p)^{1/p} \|_X^p 
\allowdisplaybreaks\\
&\ge&
2-2\eps - \|(g^p- T(h)^p)^{1/p} \|_X^p .
\eea
Now notice that the inequality 
$$
\| |g^p-T(h)^p|^{1/p} \|^p_X \le \| g-T(h)\|^p + p(2k(p))^{p/p'} \| g-T(h)\|,
$$
holds (Lemma~2.4 in \cite{EnrDir2}), 
where $k(p)$ is defined to be  $1$ if $p\ge p'$ and
$k(p)= 2^{(p'/p) - 1}$ if $p<p'$;
it can be proved using the pointwise estimates given 
in \cite[Section~2.2]{libro}, H\"older's inequality for Banach
function spaces and the 
constant~$1$ $p$-convexity of~$X$.
Therefore, we obtain that
\bea
2- 2 \varepsilon - \|g - T(h)\|_X^p 
\ge 2 -2 \varepsilon- \varepsilon^p - p(2k(p))^{p/p'} \varepsilon.
\eea
Since this holds for every $\varepsilon >0$,
we have proved the claim.

For finishing the proof let $\varepsilon >0$ and take into
account that we are assuming that the norm closure
$\overline{T(B_X)}=K$ is a weakly compact set, and therefore it is the
closed convex hull of its strongly exposed points. Thus, there is a
strongly exposed point $f_0 \in 
K$ such that $1- \varepsilon/2 < \|f_0\| \le 1$. Then there is a
slice $S$ such that $T(S) \subset B_\varepsilon(f_0)$ (see the
proof of this for example in 
\cite[Theorem~11.50]{abraali} or \cite{ams2000}). 
Since $B_{\varepsilon/2}(f_0) \subset
B_\varepsilon(f_0/\|f_0\|)$, the claim gives the result. Finally,
note that the space $X$ cannot be reflexive, for  otherwise the
operator $-\Id$ would satisfy the $p$-Daugavet equation.
\end{proof}

A relevant consequence of Theorem~\ref{reflexive} is that no $L^p(\mu)$
space for $1<p<\infty$ satisfies the $p$-Daugavet property; recall
that for $p=1$ and an atomless measure  $\mu$ the space
$L^1(\mu)$ has the Daugavet property (see
\cite[Theorem~3.2]{abraali0} 
or \cite[Example, p.~858]{ams2000}). 
Here is another corollary.

\begin{cor}\label{cor3.2}
Let $1<p<\infty$. If a constant~$1$ $p$-convex  Banach function space has the
$p$-Daugavet property, then every $p$-concave operator satisfies
the $p$-Daugavet equation. 
\end{cor}

\begin{proof}
Each such 
operator factorizes through an $L^p$ space (see
\cite[Corollary~1.d.12]{lint}), 
and so it is weakly compact. 
\end{proof}

Related arguments can
also be used for proving that there are no ideals in such spaces
being isomorphic to $L^p$ spaces.  

\begin{cor}
In a constant~$1$ $p$-convex Banach
function space with the
$p$-Daugavet property, $p>1$, there are no $p$-concave  band
projections. 
\end{cor}

\begin{proof}
Assume that $Q\dopu X \to X$ is such a
projection; there is another disjoint projection $P$ such that
$\Id= Q+P$. Then
\bea
\sup_{f \in B_X}\|(f^p-Q(f)^p)^{1/p}\|
&=&
\sup_{f \in B_X}\|((P(f)+Q(f))^p-Q(f)^p)^{1/p}\| \\
&=& 
\sup_{f \in B_X} \|(P(f)^p+Q(f)^p-Q(f)^p)^{1/p}\| \\
&=& 
\sup_{f \in B_X} \| P(f)\| \le 1 < (1+ \|Q\|^p)^{1/p}
\eea
which contradicts Corollary~\ref{cor3.2},
since $-Q$ is $p$-concave.
\end{proof}

Moreover, the same computation gives that in a constant~$1$ $p$-convex
Banach function space with the $p$-Daugavet property there are no
projection bands 
being isomorphic to $L^p$ spaces or to any reflexive Banach
space. Finally, note that if $X$ is order continuous, then every ideal
is the range of a positive contractive projection, so there are no
ideals isomorphic  to $L^p$ spaces as Banach spaces  (see
\cite[Proposition~1.a.11]{lint}).

Let us finish the paper with a suitable version of the non-existence of
unconditional bases for Banach spaces 
with the $p$-Daugavet property. 
We recall that no Banach space with the \DP\ has an unconditional
basis \cite{kadets}
and does not even embed into a space with an unconditional basis 
\cite{ams2000}.
Let $X$ be a g.f.s.\ and consider an
unconditional basis $\mathbf B:=\{e_n\dopu n \in \mathbb N\}$ with
projections $P_A$, where $A$ is a finite subset of $\mathbb N$. Write
$Q_A$ for the complementary projection $\Id-P_A$.  
We say that  $\mathbf B$ satisfies a lower $p$-estimate if there is a
constant $k>0$ such that  
$$
\|Q_A\| \ge \sup_{f \in B_{X}} \| (k f^p - P_A(f)^p)^{1/p} \|.
$$
Notice that every unconditional basis satisfies a lower $1$-estimate
with the constant $k=1$. (Indeed, this is so for every Schauder
basis.) 
Also, if the natural lattice structure associated to an unconditional
basis is considered and $X$ becomes a g.f.s.\ over the 
counting measure, the basis satisfies a lower $p$-estimate for every
$1 \le p < \infty$, also with the constant $k=1$. None of these cases
can occur if $X$ is a Banach space with the Daugavet property or a
constant~$1$ $p$-convex g.f.s.\ with the $p$-Daugavet property,
respectively. 
The following result generalizes these examples. 

\begin{proposition}\label{uncbasis}
Let $X$ be a constant~$1$ $p$-convex g.f.s.\  that has the
$p$-Daugavet property. Then $X$ does not have an  unconditional basis
with a lower $p$-estimate.  
\end{proposition}

\begin{proof}
Let $\mathbf B:=\{e_n\dopu n \in \mathbb N\}$ be an unconditional basis
for $X$ with a lower $p$-estimate. For every finite 
subset $A \subset \mathbb{N}$ denote by $P_A$ and $Q_A$ the
corresponding projections on the subspaces generated by the
elements of the basis with subscripts in $A$ and $\mathbb{N}
\setminus A$, respectively. Since $\Id= P_A + Q_A$, $\mathbf B$ has a lower $p$-estimate and by Theorem~\ref{reflexive} $-P_A$ satisfies the
$p$-Daugavet equation, there is a constant $k >0 $ such that
$$
\|Q_A \| \ge \sup_{f \in B_X} \| |k f^p + (-P_A(f))^p|^{1/p} \|_X =
(k^{1/p} + \|P_A\|^p)^{1/p}.
$$
If we define $V = \sup \{ \|P_A\|\dopu  A$ finite$\}$ and
$W = \sup \{ \|Q_A\|\dopu  A $ finite$\}$, clearly $W \le V$.
Since by the inequalities above $W \ge (k^{1/p}+ V^p)^{1/p}$, we obtain
that $V=W= \infty$, a contradiction with the fact that $\mathbf B$
is an unconditional basis.
\end{proof}

\end{document}